\documentclass{amsart}
\usepackage{amsfonts}
\usepackage{amsmath}
\usepackage{amssymb}
\usepackage{amsthm}
\usepackage[pagewise,mathlines]{lineno}
\usepackage{fullpage}
\usepackage{graphicx}
\usepackage{amsfonts}
\usepackage[stable]{footmisc}
\usepackage{xcolor}
\usepackage{soulutf8}

\numberwithin{equation}{section}

\newtheorem{theorem}{Theorem}[section]
\newtheorem{proposition}[theorem]{Proposition}
\newtheorem{lemma}[theorem]{Lemma}
\newtheorem{corollary}[theorem]{Corollary}

\newtheorem{remark}[theorem]{Remark}
\newtheorem{example}[theorem]{Example}

\theoremstyle{definition}
\newtheorem{definition}[theorem]{Definition}

\begin{document}

%
%

\title[On the Number of Weighted Zero-sum Subsequences]{On the Number of Weighted Zero-sum Subsequences}

\author{A.  Lemos, B.K. Moriya, A.O. Moura and A.T. Silva$^{\ast}$}
\thanks{$\ast$ The authors were partially supported by FAPEMIG APQ-02546-21 and RED-00133-21.}
\address{Departamento de Matem\'{a}tica, Universidade Federal de Vi\c cosa, Vi\c cosa-MG, Brazil}

\email{abiliolemos@ufv.br\\bhavinkumar@ufv.br\\allan.moura@ufv.br}
\email{anderson.tiago@ufv.br}



\begin{abstract}
Let $G$ be a finite additive abelian group with exponent $d^kn, d,n>1,$ and $k$ a positive integer. For $S$ a sequence over $G$ and $A=\{1,2,\ldots,d^kn-1\}\setminus\{d^kn/d^i:i\in[1,k]\}, $ we investigate the lower bound of the number $N_{A,0}(S)$, which denotes the number of $A$-weighted zero-sum subsequences of $S.$ In particular, we prove that $N_{A,0}(S)\ge 2^{|S|-D_A(G)+1},$ where $D_A(G)$ is the $A$-weighted Davenport Constant. We also characterize the structures of the extremal sequences for which equality holds for some groups.
\end{abstract}

\maketitle

\section{Introduction}
\hspace{0.6cm}Let $G$ be a finite additive abelian group with exponent
$n$ and $S$ be a sequence over $G$. The enumeration of subsequences
with certain prescribed properties is a classical topic in Combinatorial
Number Theory going back to Erd\H{o}s, Ginzburg and Ziv (see \cite{EGZ,Ger1,Ger2})
who proved that $2n-1$ is the smallest integer, such that every sequence
$S$ over a cyclic group $C_{n}$ has a subsequence of length $n$
with zero-sum. This raises the problem of determining the smallest
positive integer $l$, such that every sequence $S=g_{1}\cdots g_{l}$
has a nonempty zero-sum subsequence. Such an integer $l$ is called
the {\it Davenport constant of $G$} (see \cite{Dav, OlsonI}), denoted by $D(G)$,
which is still unknown for wide class of groups. In an analogous manner, for a nonempty subset $A\subset \mathbb{Z}$, Adhikari \textit{et. al.} defined, see \cite{Adh1}, an $A$-weighted Davenport constant, denoted by $D_A(G)$, to be a smallest $t\in \mathbb{N}$ such that every sequence $S$ over $G$ of length $t$ has nonempty $A$-weighted zero-sum subsequence.

For any $g$ of $G$, let $N_{A,g}(S)$ (when $A=\{1\}$ we write
$N_{g}(S)$) denote {\it the number of weighted subsequences} $T=\prod_{i\in I}g_{i}$
of $S=g_{1}\cdots g_{l}$ such that $\sum_{i\in I}a_{i}g_{i}=g$,
where $I\subseteq\left\{ 1,\ldots,l\right\} $ is a nonempty subset
and $a_{i}\in A$. In 1969, Olson, see \cite{OlsonII}, proved
that $N_{0}(S)\geq2^{|S|-D(G)+1}$ for every sequence $S$ over $G$
of length $|S|\geq D(G)$. Subsequently, several authors, including
\cite{bal,Bia,Cao,Fur,Gao2,Gao3,Gao5,Gao6,Gry1,Gry2,Gui,kis} obtained a huge variety of results on the number of subsequences with
prescribed properties. In 2011, Chang {\it et al.}, see \cite{Chang}, found the lower bound
of $N_{g}(S)$ for any arbitrary $g$ and classify the extremal sequences for $|G|$ odd. Recently, Lemos {\it et al.}, see \cite{losm}, found the lower bound
of $N_{A,0}(S)$ for $A=\{1,\dots,n-1\}$ and classify the extremal sequences for $|G|$ odd. Here we prove that $N_{A,0}(S)\geq2^{|S|-D_{A}(G)+1}$, when $A=\{1,2,\ldots,d^kn-1\}\setminus\{d^kn/d^i:i\in[1,k]\},$ where $k$ is a positive integer. Besides, we classify the sequences such that $N_{A,0}(S)=2^{|S|-D_{A}(G)+1}$, where $G=H\oplus C_{d^kn}^r,$ with $n$ odd, $\exp(H)\mid d^k,$ $\gcd(d,n)\leq d-1$ and $d^kn\geq 6$.


\section{Notations and terminologies}

\global\long\def\labelenumi{(\roman{enumi})}

\hspace{0.6cm} In this section, we will introduce some notations and terminologies. Notations and terminologies are in accordance with \cite{losm}. Let $\mathbb{N}_{0}$ be {\it the set
of non-negative integers}. For integers $a,b\in\mathbb{N}_{0}$, we
define $[a,b]=\left\{ x\in\mathbb{N}_{0}:a\leq x\leq b\right\} $.

For a sequence 
\[
S=\prod_{i=1}^{m}g_{i}\in\mathcal{F}\left(G\right),
\]
where $\mathcal{F}\left(G\right)$ is the {\it free abelian monoid} with
basis $G$, a \textit{subsequence} $T=g_{i_{1}\cdots}g_{i_{k}}$ of $S$, with
$I_{T}=\{i_{1},\ldots,i_{k}\}\subseteq[1,m]$ is denoted by $T|S$;
we identify two subsequences $S_{1}$ and $S_{2}$ if $I_{s_{1}}=I_{s_{2}}$.
Given subsequences $S_{1},\ldots,S_{r}$ of $S$, we define \textit{gcd$(S_{1},\ldots,S_{r})$} to be the sequence indexed by 
$I_{S_{1}}\cap\cdots\cap I_{S_{r}}.$ We say that two subsequences
$S_{1}$ and $S_{2}$ are {\it disjoint} if $(S_{1},S_{2})=\lambda$, where
$\lambda$ refers to the \textit{empty sequence}. If $S_{1}$ and $S_{2}$
are disjoint, then we denote by $S_{1}S_{2}$ the subsequence with
set index $I_{s_{1}}\cup I_{s_{2}}$; if $S_{1}|S_{2}$; we denote
by $S_{2}S_{1}^{-1}$ the subsequence with set index $I_{s_{2}}\setminus I_{s_{1}}$.
Moreover, we define 
\begin{enumerate}
\item $\left|S\right|=m$ {\it the length of $S$}.
\item an {\it $A$-weighted sum} is a sum of the form $\sigma^{\mathbf{a}}\left(S\right)=\sum_{i=1}^{m}a_{i}g_{i}$, with fixed
$\mathbf{a}=a_{1}\cdots a_{m}\in\mathcal{F}(A)$, where $\mathcal{F}\left(A\right)$
is the free abelian monoid with basis $A$. When $A=[1,n-1]$, we call $S$ a {\it fully weighted sequence}.  
\item $\sum_{A}\left(S\right)=\left\{ \sum_{i\in I}a_{i}g_{i}:\emptyset\neq I\subseteq\left[1,m\right]\mbox{ and }a_{i}\in A\right\} $, a {\it set of nonempty $A$-weighted subsums of $S$}. 
\end{enumerate}
According to the above definitions, we adopt the convention that $\sigma^{\mathbf{a}}\left(\lambda\right)=0$,
for any $\mathbf{a}\in\mathcal{F}(A)$. For convenience, we define
$\sum_{A}^{\bullet}\left(S\right)=\sum_{A}\left(S\right)\cup\left\{ 0\right\} $.

The sequence $S$ is called 
\begin{enumerate}
\item an {\it $A$-weighted zero-sum free sequence} if $0\notin\sum_{A}\left(S\right)$ and
\item an {\it $A$-weighted zero-sum sequence} if $\sigma^{\mathbf{a}}\left(S\right)=0$
for some $\mathbf{a}\in\mathcal{F}(A)$. 
\end{enumerate}
When $A=\{1\},$ we call $S$ {\it zero-sum free sequence} and {\it zero-sum sequence}, respectively.
For an element $g\in G$, let 
\[
N_{A,g}\left(S\right)=\left|\left\{ I\subseteq\left[1,m\right]:\sum_{i\in I}a_{i}g_{i}=g,\:a_{i}\in A\right\} \right|
\]
denote {\it the number of subsequences $T$ of $S$} with $\sigma^{\mathbf{a}}\left(T\right)=g$
for some $\mathbf{a}\in\mathcal{F}(A)$. 
\begin{definition}
Let $n$ be the exponent of $G$, $g\in G$, $A\subseteq\mathbb{Z}\backslash\left\{ kn:k\in\mathbb{Z}\right\} $ and $S\in\mathcal{F}\left(G\right)$. We say $S$ is {\it $g$-complete sequence with weight in $A$} if $N_{A,g}\left(S\right)\geq2^{\left|S\right|-D_{A}\left(G\right)+1}$. We call $S$ an {\it extremal $g$-complete sequence with respect to $A$} if $N_{A,g}\left(S\right)=2^{\left|S\right|-D_{A}\left(G\right)+1}$.
Let us denote $C_{A,g}\left(\mathcal{F}\left(G\right)\right)$ as
the {\it set of all $g$-complete sequences with respect to $A$} and $EC_{A,g}\left(\mathcal{F}\left(G\right)\right)$
as the {\it set of all extremal $g$-complete sequences with respect to $A$}. 
\end{definition}
\vspace{0cm}

\begin{definition}
Let $n$ be the exponent of $G$ and $A\subseteq\mathbb{Z}\backslash\left\{ kn:k\in\mathbb{Z}\right\} $.
We say $G$ is a {\it $0$-complete group with respect to $A$} if $\mathcal{F}\left(G\right)=C_{A,0}\left(\mathcal{F}\left(G\right)\right)$. 
\end{definition}
When $A=\left\{ 1\right\} $, Olson \cite{OlsonII} proved that all finite
abelian groups are $0$-complete with respect to $A$. Chang {\it et al.} \cite{Chang}
proved, that, when $A=\left\{ 1\right\} $, if $g\in\sum_{A}^{\bullet}\left(S\right)$,
then $S\in C_{A,g}\left(\mathcal{F}\left(G\right)\right)$ and, if
$S$ is extremal $h$-complete sequence with respect to $A$ for some
$h\in G$, then $S$ is $g$-complete sequence with respect to $A$
for all $g\in G$. Moreover, they classified the sequences in $EC_{A,0}\left(\mathcal{F}\left(G\right)\right)$
when $G$ is a group of odd order.

\begin{remark}
 Take an $A$-weighted zero-sum free sequence $U$
over $G$ with $|U|=D_{A}(G)-1$. Thus, for $S=U0^{|S|-D_{A}(G)+1}$
and for any $g\in\sum_{A}^{\bullet}\left(U\right)$, we have $S\in C_{A,g}\left(\mathcal{F}\left(G\right)\right)$
and $S\in EC_{A,0}\left(\mathcal{F}\left(G\right)\right)$.
\end{remark}

We write a finite abelian group $G$ as direct sum $G=H\oplus C_{n}^{r}$
, where $C_{n}^{r}$ denotes $r$ copies of the cyclic group of order $n$ denoted by $C_{n}$ and $H=C_{n_{1}}\oplus\cdots\oplus C_{n_{t}}$ with $1<n_{1}|n_{2}|\cdots|n_{t}|n=exp(G)$
and $n_{t}<n$. 

We have some auxiliary results, which are
as follows. 
\begin{lemma} \label{L1}
\label{lem:13}{[}Theorem 5.2 \cite{marc}{]} Let $G=H\oplus C_{n}^{r}$, where $H=C_{n_{1}}\oplus\cdots\oplus C_{n_{t}}$ with
$1<n_{1}|n_{2}|\cdots|n_{t}|n=exp(G)$ and $n_{t}<n$. Then, $D_{A}(G)=r+1$. 
\end{lemma}

A subsequence $T$ of $S$ is called an {\it extremal $A$-weighted zero-sum free subsequence} if $|T|=D_A(G)-1$ and $T$ is $A$-weighted zero-sum free.

It is worth mentioning the following important result for the fully weighted 0-complete sequences, which was proved in \cite{losm}.

\begin{theorem}\label{T1}
\label{prop:cpr} All finite abelian group $G$ with exponent $n$
is $0$-complete with respect to $A=\left[1,n-1\right]$. 
\end{theorem}


In \cite{losm} the authors conjectured that Theorem \ref{T1} holds for any $A.$ In the Section \ref{s3}, we proved that such a theorem is true for $G = H\oplus C_{d^kn}^r,$ with $n$ odd, $\exp(H)|d^k, \gcd(d, n)< d- 1,$  $d^kn<6$ and $A=\{1,2,\ldots,d^kn-1\}\setminus\{d^kn/d^i:i\in[1,k]\},$ where $k$ is a positive integer.

\section{Lower bound}\label{s3}

We start this section by presenting an important theorem.

\begin{theorem}[Adhikari {\it et al}.,Theorem 4.1, item (i) \cite{AGS}]\label{adhi}
Let $G$ be a finite and nontrivial abelian group and let $S\in\mathcal{F}(G)$ be a sequence. If $|S|\geq\log_2|G|+1$ and $G$ is not an elementary $2$-group, then $S$ contains a proper, nontrivial $\{\pm1\}$-weighted zero-sum subsequence.
\end{theorem}

To find the lower bound for $N_{A,0}(S),$ with $S\in\mathcal{F}(G),$ we used the value of $D_A(G).$

\begin{theorem}\label{t2}
Let $G=H\oplus C_{d^kn}^r$, where $\exp(H)\mid d^k,$ $\gcd(d,n)\leq d-1$ and $d^kn\geq 6,$ where $k$ is a positive integer. Then $D_A(G)=r+1$, for $A=\{1,2,\ldots,d^kn-1\}\setminus\{d^kn/d^i:i\in[1,k]\}$.
\end{theorem}
\begin{proof}
Since the canonical sequence $\prod_{i=1}^r e_i\in \mathcal{F}(C_{d^kn}^r)$ does not have zero-sum subsequence with respect to weights in $A$, $D_A(G)\geq r+1.$ Let $S=(h_i,g_i)_{i=1}^{r+1}\in\mathcal{F}(G)$. Consider a canonical homomorphism $\phi:G\rightarrow C_n^r$. Let $A'=\{1,2,\ldots,n-1\}.$ Since $D_{A'}(C_n^r)=r+1,$ by Lemma \ref{L1}, we get a non-empty subsequence $T=(g_{i_k})_{k=1}^t$ of $S$ such that $\sum_{j=1}^ta_j\phi(g_{i_j})=0\in C_n^r$, where $a_j\in A',\forall j$. Hence, using the fact that $\exp(H)\mid d^k$ we have, $\sum_{j=1}^td^ka_j(h_{i_j},g_{i_j})=0\in G$ (Note that $\phi(g)\equiv g\pmod n$, which as a result gives, $d^k\phi(g)\equiv d^k\cdot g\pmod {d^k\cdot n}$). Since $\gcd(d,n)\leq d-1$, it follows that $d^ka_j\in A,\forall j$, which proves the theorem.
\end{proof}

The hypothesis $d^kn\ge6$ in Theorem \ref{t2} is necessary, as on the contrary we have the following proposition.  

\begin{proposition}\label{p5}
If $G=C_2^s\oplus C_4^r$, then $D_A(G)=2r+s+1$ for $A=\{1,3\}=\{1,-1\}$.
\end{proposition}
\begin{proof}
This upper bound is a immediate consequence of Theorem \ref{adhi}. For the lower bound we observe that the sequence $S=\prod_{i=1}^{s+r}e_i\prod_{i=s+1}^{s+r}2e_i$ does not have $\{1,-1\}$-weighted zero-sum subsequece, where $\{e_1,\dots,e_{s+r}\}$ is the canonical base of $G.$ 
\end{proof}

\begin{remark}
Note that, if $B\subset A$ then $D_A(G)\le D_B(G)$. 
\end{remark}

As a consequence of Theorem \ref{t2}, we get the following corollary.
\begin{corollary}\label{c1}
Let $G = H\oplus C_{d^kn}^r$, where $\exp(H)\mid d^k,$ $\gcd(d,n)\leq d-1$ and $d^kn\geq 6,$ where $k$ is a positive integer. Then $D_A(G)=r+1$, for all $A$ containing $B=\{1,2,\ldots,d^kn-1\}\setminus\{d^kn/d^i:i\in[1,k]\}$.
\end{corollary}

One can easily see that Theorem \ref{t2} does not hold true if $\exp(H)\nmid d^k$, in fact, the next proposition provides infinitely many examples such that $D_{A\setminus\{n\}}(G)\neq D_{A}(G)$.

\begin{proposition}\label{T4}
Let\ $G=L\oplus C_n\oplus C_{2n}^r$, where $n>1$ an odd number and $L=C_{n_1}\oplus\dots\oplus C_{n_t}$ with $n_1|n_2|\cdots|n_t|n$. Then, $D_{A\setminus\{n\}}(G)\ge r+2$.
\end{proposition}
\begin{proof}
Since $n$ is an odd number one can easily prove that $(e_{t+1}+e_{t+2})(e_{t+1}+ne_{t+2})$ is zero-sum free with respect to weights in $A\setminus\{n\}$ and which in turn gives rise to a $A\setminus\{n\}$-weighted zero-sum free sequence $(e_{t+1}+e_{t+2})(e_{t+1}+ne_{t+2})\prod_{i=3}^{r+1}e_{t+i}$.
\end{proof}

%


Theorem \ref{0-c} below provides one more case for which Conjecture 4.3 of \cite{losm} holds.

\begin{theorem}\label{0-c}
Let $G=H\oplus C_{d^kn}^r$, where $\exp(H)\mid d^k,$  $  \gcd(d,n)\leq d-1$ and $d^kn\geq 6,$ where $k$ is a positive integer. Then $G$ is $0$-complete with respect to $A=\{1,2,\ldots,d^kn-1\}\setminus\{d^kn/d^j:j\in[1,k]\}.$
\end{theorem}

\begin{proof}
Let $S\in\mathcal{F}(G)$ be a sequence. According to Corollary \ref{c1}, we can write $D_{A}(G)=r+1$. If $|S|\leq r$, then $N_{A,0}(S)\geq1\geq 2^{|S|-r}$.
If $|S|=r+1$, then there is an $A$-weighted zero-sum nonempty subsequence $T$ of $S$. Thus, $N_{A,0}(S)\geq 2=2^{|S|-r}$. 

Suppose now $r+1<|S|$. Let $S=TW\in\mathcal{F}(G)$ be a sequence such that $T$ is a maximal $A$-weighted zero-sum free with $|T|\le r$ or $T=\lambda.$ 

Then, for each element $g|W$, we have two possibilities:

\textbf{a)} If $o(g)\in A$, then $g$ is an $A$-weighted zero-sum subsequence.
 
\textbf{b)} If $o(g)\not\in A$  for $j=1,\dots k$, then $Tg$ has an $A$-weighted zero-sum subsequence with $g$ being one of its elements.

In both possibilities, there is $V|Tg$, such that $V$ is an $A$-weighted zero-sum subsequence
whose coefficient of $g$ is $a_{g}\in A$. Then, $a_{g}g$ is an $A$-weighted sum
of some subsequence of $T$:
\[a_{g}g={\sum}_{i\in I_g}a_i g_i; I_g\subset I_T\mbox{ and } a_i\in A.\] 
Thus, for every nonempty subsequence $U$ of $W,$ we have
\[d^k{\sum}_{g|U}a_{g}g={\sum}_{g|U} d^k{\sum}_{i\in I_g}a_i g_i={\sum}_{i\in I_{V'}}d^kb_i g_i; I_{V'}\subset I_T, \] with $d^kb_i\pmod{d^kn}\in A$ or $d^kb_i\equiv0\pmod{d^kn}$,
i. e., the $A$-weighted sum $d^k{\sum}_{g|U}a_{g}g$ is an $A$-weighted sum of some subsequence $V'=V_{U}$ of $T$. Therefore, $UV_{U}$ is an $A$-weighted zero-sum subsequence of $S$. Notice that if $V_{U}=\lambda$, then $U$ is an $A$-weighted zero-sum subsequence. Therefore, if we include the empty subsequence, we obtain
a minimum of $2^{|W|}=2^{|S|-|T|}$ distinct $A$-weighted zero-sum subsequences
of $S$. This proves that $N_{A,0}(S)\ge 2^{|S|-r}$. 
\end{proof}

\section{Characterization of extremal $0$-complete sequences}\hspace{0.6cm}\label{CE0cs} 

We shall start by mentioning one of the main results obtained in the case which $\exp(G)$ is an odd positive integer (see Theorem 4.2 of \cite{losm}).

\begin{theorem}
Let $G$ be a finite abelian group with $exp(G)=n$ an odd number. If $S\in EC_{A,0}\left(\mathcal{F}\left(G\right)\right)$, with $A=[1,n-1]$, $0\nmid S$ and $o(g)=n$ for all $g|S$, then $r\leq\left|S\right|\leq2r$ and there is $T=\prod_{i=1}^{r}g_i$ an extremal $A$-weighted zero-sum free, such that 
\begin{equation}
S=T\prod_{j=1}^{k}h_{j},
\end{equation}
 where $k\in\left[1,r\right]$, $b_{j}h_{j}=\sum_{i\in I_{j}}a_{i}g_{i}$
with $a_{i},b_{j}\in A$, $I_{j}\subset\left[1,r\right]$ and $I_{j}$'s
are pairwise disjoint ($I_{j}=\emptyset$ for all $j$ implies that
$S=T$).
\end{theorem}

In this section, our aim is to prove a variant of the result above in case $G = H\oplus C_{d^kn}^r$ be a finite abelian group, where $k$ is a positive integer, $\exp(H)\mid d^k,$ $n$ is an odd number, $\gcd(d,n)\leq d-1,d^kn\geq 6$ and $A=\{1,2,\ldots,d^kn-1\}\setminus\{d^kn/d^j:j\in[1,k]\},$ which will be established in Theorem \ref{thm11}. 


First, we consider a modification of the Proposition 4.1 (see \cite{losm}), which will be the main tool to prove the Theorem \ref{thm11}.

As $N_{A,0}\left(S\right)=2N_{A,0}\left(S0^{-1}\right)$ and $N_{A,0}\left(S\right)=2N_{A,0}\left(Sg^{-1}\right)$,
if $o(g)\in A,$ it suffices to consider sequences $S$, such that $0\nmid S$ and $o(g)\not\in A$ for all $g|S$. 

\begin{proposition}
\label{formadasequencia}
Let $G = H\oplus C_{d^kn}^r$ be a finite abelian group where $\exp(H)\mid d^k,$ \linebreak $\gcd(d,n)\leq d-1$ and $d^kn\geq 6,$ where $k$ is a positive integer. If $S\in EC_{A,0}\left(\mathcal{F}\left(G\setminus \{0\}\right)\right)$, with \linebreak $A=\{1,2,\ldots,d^kn-1\}\setminus\{d^kn/d^j:j\in[1,k]\}$ and $o(g)\not\in A$ for all $g|S$, then $r\leq\left|S\right|$ and there is $T=\prod_{i=1}^{r}g_i$ an extremal $A$-weighted zero-sum free such that 
\begin{equation}
S=T\prod_{j=1}^{\nu}h_{j},
\end{equation}
 where $\nu\in\left[1,r\right]$, $b_{j}h_{j}=\sum_{i\in I_{j}}a_ig_{i}$
with $a_i, b_{j}\in A$, $I_{j}\subset\left[1,r\right]$.

\end{proposition}
\noindent The proposition above is a mere consequence of $D_A(G) = r+1$.

Let us see below an example where we show an extreme sequence with respect to $N_{A,0}(S)$ for a group of order $72.$

\begin{example}
Let $S=e_2e_3(2e_2)(2e_3)$ be a sequence over $G=C_2\oplus C_6^2,$ where $\{e_1,e_2,e_3\}$ is the canonical basis of $G.$ Note that, $D_A(G)=3$ where $A=\{1,2,4,5\}$, $|S|=4=2(D_A(G)-1),$ and $N_{A,0}(S)=2^{|S|-D_A(G)+1}=2^2=4$. In this case, $T=e_2e_3$ is an extremal $A$-weighted zero-sum free.
\end{example}

The example above motivates us to establish the theorem below. 

\begin{theorem}\label{thm11}
Let $G = H\oplus C_{d^kn}^r$ be a finite abelian group where $n$ is an odd number,\linebreak $\exp(H)\mid d^k,$ $\gcd(d,n)\leq d-1$ and $d^kn\geq 6,$ where $k\in\mathbb{N}.$ If $S\in EC_{A,0}\left(\mathcal{F}\left(G\setminus \{0\}\right)\right)$, with \linebreak $A=\{1,2,\ldots,d^kn-1\}\setminus\{d^kn/d^j:j\in[1,k]\}$ and $o(g)\not\in A$ for all $g|S$, then $r\leq\left|S\right|\leq2r$ and there is $T=\prod_{i=1}^{r}g_i$ an extremal $A$-weighted zero-sum free such that 
\begin{equation}
S=T\prod_{j=1}^{\nu}h_{j},
\end{equation}
 where $\nu\in\left[1,r\right]$, $b_{j}h_{j}=\sum_{i\in I_{j}}a_ig_{i}$
with $a_i, b_{j}\in A$, $I_{j}\subset\left[1,r\right]$ and $I_{j}$'s
are pairwise disjoint ($I_{j}=\emptyset$ for all $j$ implies that
$S=T$).
\end{theorem}

\begin{proof}
Let $S$ be a sequence over $G\setminus \{0\}$, $o(g)\not\in A$ for all $g|S$ and $N_{A,0}\left(S\right)=2^{\left|S\right|-D_{A}\left(G\right)+1}=2^{|S|-r}.$ We know, by Proposition \ref{formadasequencia}, that $S=T\prod_{j=1}^{\nu}h_{j}$ where $\nu\in\mathbb{N}_0$, $b_{j}h_{j}=\sum_{i\in I_{j}}a_{i}g_{i}$ with $a_{i},b_{j}\in A$, $I_{j}\subset\left[1,r\right]$ and $T=\prod_{i=1}^{r}g_i$ is an extremal $A$-weighted zero-sum free. 

Now, we will prove that the $I_{j}$'s are pairwise disjoint. If $\left|S\right|=D_{A}\left(G\right)-1=r$, then $N_{A,0}(S)=1$, $I_{j}=\emptyset$ for all $j\in [1,\nu]$ and $S=T$. Suppose that $\left|S\right|=D_{A}\left(G\right)=r+1$ then, $I_{j}\neq\emptyset$ for only one $j$, $N_{A,0}(S)=2$ and
$S=Th_{j}$. Finally, suppose $S=T\prod_{j=1}^{\nu}h_{j}$
with $\nu\geq2$ and $I_{j_1}\cap I_{j_2}\neq\emptyset$ for some $j_1,j_2\in\left[1,\nu\right]$, with $j_1\neq j_2$ and
where \[a_{j_1}h_{j_1}=\sum_{i\in I_{j_1}}a_{i}g_{i} \mbox{ and } a_{j_2}h_{j_2}=\sum_{i\in I_{j_2}}b_{i}g_{i}\] with $a_{j_1},a_{j_2},a_i,b_{i}\in A$, since $D_A(G)=r+1$.

By hypothesis we have $\tbinom{\nu}{0}+\tbinom{\nu}{1}+\cdots+\tbinom{\nu}{\nu}=2^{\nu}=2^{\left|S\right|-r}$
$A$-weighted zero-sum subsequences of $S,$ which can be obtained as in the proof of Theorem \ref{0-c}. Since $I_{j_1}\cap I_{j_2}\neq\emptyset$, we have $I_{x}, I_{y}\subset I_{j_1}\cup I_{j_2}$ such that 
\begin{equation}\label{eq.7}
d^k(a_{j_1}h_{j_1}+a_{j_2}h_{j_2})=d^k\left(\sum_{i\in I_{x}}c_{i}g_{i}\right) \pmod{d^kn}
\end{equation}
and 
\begin{equation}\label{eq.8}
d^k(a_{j_1}h_{j_1}-a_{j_2}h_{j_2})=d^k\left(\sum_{i\in I_{y}}d_{i}g_{i}\right)\pmod{d^kn}.  
\end{equation}

Since $d^ka_{j_1},d^ka_{j_2}\pmod{d^kn} \in A$ one can easily verify that $d^kc_i\equiv0\pmod{d^kn}$ or $d^kc_i\pmod{d^kn}\in A$ and $d^kd_i\equiv0\pmod{d^kn}$ or $d^kd_i\pmod{d^kn}\in A.$ If $d^kc_i\equiv0\pmod{d^kn},$ for all $i\in I_{x}$ and $d^kd_i\equiv0\pmod{d^kn}$ for all $i\in I_{y},$ then $d^k(a_{j_1}h_{j_1}+a_{j_2}h_{j_2})=0$ and $d^k(a_{j_1}h_{j_1}-a_{j_2}h_{j_2})=0.$ But, this implies that $2d^{k}a_{j_2}h_{j_2}=0,$ and hence $n|a_{j_2}$ ($o(h_{j_2})=d^kn$ and $n$ is odd), which is a contradiction since $d^{k}a_{j_2}\not\equiv 0\pmod{d^kn}.$ Therefore, $I_{x}\neq\emptyset$ or $I_y\neq\emptyset$.

If $I_x\neq I_y$, then there is a new $A$-weighted zero-sum subsequence of $S$ and therefore $N_{A,0}\left(S\right) >2^{|S|-r}$, which is a contradiction.
Now, suppose that $I_{x}=I_y$. Consider $g_l|\prod_{i\in I_{j_1}\cap I_{j_2} }g_{i}$ (observe that $d^kc_l\not\equiv0\pmod{d^kn}$ and $d^kd_l\not\equiv0\pmod{d^kn}$ in \eqref{eq.7} and \eqref{eq.8}) and take $T'=\left(\prod_{i=1}^{r+1}g_i\right)g_l^{-1}$, where $g_{r+1}=h_{j_2}$. If $T'$ is not an extremal $A$-weighted zero-sum free, then there is $\bar{I}_{j_2}\subset \left[1,r+1\right]\backslash \left\{l\right\}$ such that $z_{j_2}h_{j_2}=\sum_{i\in \bar{I}_{j_2}}s_{i}g_{i}$, i.e., we can obtain a new $A$-weighted zero-sum subsequence of $S$ and thus $N_{A,0}\left(S\right) >2^{|S|-r}$, which is a contradiction. If $T'$ is an extremal $A$-weighted zero-sum free, then by Corollary \ref{c1} we have $\bar{I}_{j_1}\subset \left[1,r+1\right]\backslash \left\{l\right\}$ such that $v_{j_1}h_{j_1}=\sum_{i\in \bar{I}_{j_1}}u_{i}g_{i}$, i.e., we can obtain a new $A$-weighted zero-sum subsequence of $S$. Therefore, we have $N_{A,0}\left(S\right) >2^{|S|-r}$ again, which is a contradiction.

We observe that if $\nu>r$, then there are $I_{j_1}$ and $I_{j_2}$
with $j_1\neq j_2$, such that $I_{j_1}\cap I_{j_2}\neq\emptyset$. Therefore, $N_{A,0}\left(S\right)>2^{|S|-r}$. Thus, $r\leq|S|\leq 2r$.
\end{proof}

The example below shows a case that is not covered by hypotheses of Theorem \ref{thm11}. We believe that it is possible to obtain a similar theorem that covers this case.

\begin{example}
Let $S=e_1e_2e_3(2e_2)(2e_3)(3e_2)(3e_3)$ be a sequence over $G=C_2\oplus C_4^2,$ where $\{e_1,e_2,e_3\}$ is the canonical basis of $G.$ Note that $|S|=7=D_A(G)+1$ and $N_{A,0}(S)=2^{|S|-D_A(G)+1}=2^2=4,$ where $D_A(G)=6,$ with $A=\{1,3\},$  by Proposition \ref{p5}. In this case, $T=e_1e_2e_3(2e_2)(2e_3)$ is an extremal $A$-weighted zero-sum free.
\end{example}


\end{document}